\documentclass{article}
\usepackage{setspace} 
\usepackage[left=1.5in,right=1.5in,top=1in,bottom=1in]{geometry}
\usepackage{amsfonts}
\usepackage{amsmath}
\usepackage{amssymb}
\usepackage{amsthm}
\usepackage{mathrsfs}
\usepackage{mathtools}
\usepackage{slashed}
\usepackage{enumerate}
\usepackage{hyperref}
\usepackage{cleveref}
\usepackage{tikz,tikz-cd}
\usepackage{extpfeil}
\usepackage{hhline}
\usepackage{float}
\usepackage{thmtools}
\usepackage{thm-restate}
\usepackage{xcolor}

\usepackage{mathbbol} 

\hypersetup{
    colorlinks=false,
    linkcolor=violet,
    filecolor=magenta,      
    urlcolor=cyan,
    citecolor=red,
    }

\makeatletter
\newcommand*{\doublerightarrow}[2]{\mathrel{
  \settowidth{\@tempdima}{$\scriptstyle#1$}
  \settowidth{\@tempdimb}{$\scriptstyle#2$}
  \ifdim\@tempdimb>\@tempdima \@tempdima=\@tempdimb\fi
  \mathop{\vcenter{
    \offinterlineskip\ialign{\hbox to\dimexpr\@tempdima+1em{##}\cr
    \rightarrowfill\cr\noalign{\kern.5ex}
    \rightarrowfill\cr}}}\limits^{\!#1}_{\!#2}}}

\newcommand\QQ{\mathbb{Q}}
\newcommand\CC{\mathbb{C}}

\newcommand\RR{\mathbb{R}}
\newcommand\NN{\mathbb{N}}
\newcommand\ZZ{\mathbb{Z}}

\numberwithin{equation}{section}

\theoremstyle{plain}
\newtheorem{theorem}{Theorem}[section]

\newtheorem{proposition}[theorem]{Proposition}
\newtheorem{corollary}[theorem]{Corollary}
\newtheorem{lemma}[theorem]{Lemma}

\theoremstyle{definition}
\newtheorem{definition}[theorem]{Definition}
\newtheorem{example}[theorem]{Example}

\theoremstyle{definition}
\newtheorem{remark}[theorem]{Remark}

\newtheorem{problem}[theorem]{Problem}

\newcommand{\del}{\partial}
\newcommand{\delbar}{{\overline{\partial}}}

\newcommand{\simplex}{{\mathbf{\Delta}}}
\newcommand{\set}{{\mathbf{Set}}}
\newcommand{\bss}{{\mathbf{BsSet}}}

\renewcommand{\ss}{{\mathbf{sSet}}}

\newcommand{\Fun}{\mathbf{Hom}}

\newcommand{\BC}{{\mathrm{BC}}}
\newcommand{\A}{{\mathrm{A}}}
\newcommand{\dbar}{{\overline{d}}}

\DeclareMathOperator{\im}{{\mathrm{im}}}

\author{Jiahao Hu}
\newcommand{\Addresses}{{
  \bigskip
  \footnotesize

  \textsc{Yau Mathematical Sciences Center, Tsinghua University,
    Beijing, China 100084}\par\nopagebreak
  Email: \texttt{jiahao.hu.math@gmail.com}

}}
\date{}
\title{Bott-Chern and Aeppli Homotopy}
\begin{document}
\maketitle

\begin{abstract}
	This paper introduces Bott-Chern and Aeppli \textit{homotopy} sets for a fibrant class of bisimplicial sets and establishes their basic properties. In positive bidegrees, Bott-Chern homotopy sets carry natural monoid structures, while Aeppli homotopy sets carry natural group structures. They are related by a loop-space comparison: after a bidegree shift, the Aeppli homotopy groups of $X$ are naturally identified with the Bott-Chern homotopy monoids of the loop space of $X$. In particular, the Bott-Chern homotopy monoids of loop spaces are groups.
	
	To justify our definitions, we show that the Bott-Chern homotopy monoids of a bisimplicial abelian group are naturally isomorphic to the Bott-Chern \textit{homology} groups of its associated normalized Moore bicomplex. An analogous statement holds for Aeppli homotopy.
	\end{abstract}
\section{Introduction}
The purpose of this paper is to define and establish basic properties of Bott-Chern and Aeppli homotopy sets for a fibrant class of bisimplicial sets. On the one hand, our work is a direct generalization of Kan's combinatorial definition of simplicial homotopy groups \cite{kan1958combinatorial}. On the other hand, this work can also be regarded as an attempt to de-localize Stelzig's definition of rational\footnote{Although Stelzig works over $\CC$, the relevant results also hold over $\QQ$.} Bott-Chern and Aeppli cohomotopy groups for rational commutative bigraded bidifferential algebras \cite{stelzig2025pluripotential}.

Stelzig's work is a bigraded enhancement of Sullivan's  approach to rational homotopy theory via commutative differential graded algebras \cite{sullivan1977infinitesimal}; by the work of Quillen and Sullivan, this algebraic theory models the rational localization of the Kan-Quillen simplicial homotopy theory under suitable nilpotency and finite type hypotheses \cite{quillen1969rational}\cite{sullivan1977infinitesimal}. We expect the Bott-Chern and Aeppli homotopy sets defined here to recover Stelzig's invariants after rational localization, although we do not treat that comparison in this paper.

Our guiding analogy is the Dold-Kan correspondence \cite{kan1958functors}. For simplicial abelian groups, homotopy groups recover the homology groups of the associated Moore complexes. Iterating Dold-Kan identifies bisimplicial abelian groups with non-negatively bigraded bicomplexes. Since Bott-Chern and Aeppli homology are natural invariants of bicomplexes, bisimplicial sets provide a natural ambient category for investigation.

The main technical difficulty lies in choosing an appropriation notion of fibration which allows the two simplicial directions to be used independently. Moreover this choice must provide enough fillers to define homotopy, loop objects and products, while still including bisimplicial groups as fibrant objects so that comparison with bicomplexes remains possible. It turns out that the standard notions of fibration for bisimplicial sets, such as Reedy fibrations or Bousfield-Kan fibrations, are not tailored to this problem. So we must define our own notion of fibration and study their properties, this is treated in \Cref{section-2}.

Once this fibrancy framework is fixed, the definition of Bott-Chern and Aeppli homotopy sets is intuitive. A Bott-Chern homotopy class is represented by a map from a standard bisimplex which is trivial on the full bisimplicial boundary. An Aeppli homotopy class is represented by a map satisfying a different boundary condition, modeled on the Aeppli cycles of a bicomplex.

Now the second difficulty is to define products on these homotopy sets. In ordinary simplicial homotopy theory, the group law is produced by filling horns. Here the relevant fillers are two-directional, making the construction less obvious and quite involved. We show that Bott-Chern homotopy sets carry natural monoid structures in positive bidegrees, while Aeppli homotopy sets carry natural group structures. The well-definedness and associativity of these operations depend essentially on the chosen notion of fibration.

The resulting Bott-Chern monoids and Aeppli homotopy groups have the expected formal behavior. Aeppli homotopy is naturally identified, up to a bidegree shift, with Bott-Chern homotopy of the loop space. In particular, Bott-Chern homotopy monoids of loop spaces are groups. Fiber sequences of fibrant bisimplicial sets give exact sequences relating the two homotopy sets.

Finally we justify our defintion by showing that, for a bisimplicial group, the Bott-Chern homotopy monoids and Aeppli homotopy groups defined here agree, in suitable bidegrees, with the Bott-Chern and Aeppli homology groups of its associated Moore bicomplex.
\subsection*{Acknowledgements} The author thanks Jianfeng Lin for organizing a seminar on simplicial homotopy theory, from which many of the ideas in this paper originated, and Jonas Stelzig for helpful discussions.

\section{Bisimplicial sets}\label{section-2}
\subsection{Basic definitions}
\begin{definition}[bisimplicial set]
A bisimplicial set $X$ is a bigraded set $X_{m,n}$ for $m,n\in\NN$ such that
\begin{enumerate}[(1)]
	\item for each $n$, $X_{\bullet,n}$ is a simplicial set with face and degeneracy maps $d_i, s_i$;
	\item for each $m$, $X_{m,\bullet}$ is a simplicial set with face and degeneracy maps $\dbar_j, \overline{s}_j$; and
	\item $\{d_i, s_i\}$ commute with $\{\dbar_j, \overline{s}_j\}$.
\end{enumerate}
We will call $X_{\bullet,n}$ (resp. $X_{m,\bullet}$) the $n$-th horizontal (resp. $m$-th vertical) slice of $X$.

Equivalently, a bisimplicial set is a functor $(\simplex\times\simplex)^{op}\to \set$ where $\simplex$ is the simplex category and $\set$ is the category of sets.
\end{definition}

\begin{example}[box product]
	For simplicial sets $X,Y$, their box product $X\boxtimes Y$ is the bisimplicial set
	\[
	(X\boxtimes Y)_{m,n}=X_m\times Y_n
	\]
	with the naturally induced bisimplicial structure, i.e. $d_i=d_i^X$, $s_i=s_i^X$ and $\overline{d}_j=d_j^Y$, $\overline{s}_j=s_j^Y$. We list below some bisimplicial sets that will be frequently used throughout.
	\begin{enumerate}[(1)]
		\item $\Delta^{m,n}=\Delta^m\boxtimes\Delta^n$. By the Yoneda lemma, for a bisimplicial set $X$, elements of $X_{m,n}$ can be identified with maps $\Delta^{m,n}\to X$.
		Elements of $X_{0,0}$ will be called vertices of $X$. We will denote $\Delta^{0,0}$, or a chosen vertex $\Delta^{0,0}\to X$ by $*$.
		
		\textbf{All of the following are sub-bisimplicial sets of $\Delta^{m,n}$.}
		\item $\partial\Delta^{m,n}=\partial\Delta^m\boxtimes\Delta^n=\bigcup_{i=0}^m d_i\Delta^{m,n}$.
		\item $\delbar\Delta^{m,n}=\Delta^m\boxtimes\partial\Delta^n=\bigcup_{j=0}^n\dbar_{j}\Delta^{m,n}$.
		\item $\delta\Delta^{m,n}=\partial\Delta^{m,n}\cup \delbar \Delta^{m,n}=\partial\Delta^m\boxtimes\Delta^n\cup \Delta^m\boxtimes\partial\Delta^n$.

The category of bisimplicial sets $\bss$ has a cellular model with generating monomorphisms given by
\[
\delta\Delta^{m,n}\hookrightarrow \Delta^{m,n}.
\]
This is to say that every bisimplicial set can be constructed by iteratively attaching $\Delta^{m,n}$ along $\delta\Delta^{m,n}$. See \cite[Example 1.3.4, Theorem 1.3.8]{cisinski2019higher}.
	\item $\del\delbar\Delta^{m,n}=\del\Delta^m\boxtimes\delbar\Delta^n=\bigcup_{i,j} d_i \dbar_j\Delta^{m,n}$.
		\item $\Lambda^m_k\boxtimes\Delta^n\cup\Delta^m\boxtimes\Lambda^n_l$ where $\Lambda^m_k=\bigcup\limits_{i\neq k} d_i\Delta^{m}$. For later use, we introduce the notation
		\[
		\Lambda^{m,n}:=\Lambda^m_m\boxtimes\Delta^n\cup\Delta^m\boxtimes\Lambda^n_n.
		\]
		\item $\delta_\A\Delta^{m,n}=\Lambda^m_m\boxtimes\Delta^n\cup\Delta^m\boxtimes\Lambda^n_n\cup \del\delbar\Delta^{m,n}=\Lambda^{m,n}\cup d_m\dbar_n\Delta^{m,n}$.
	\end{enumerate}
\end{example}

\subsection{Anodyne extensions and fibrations}
Recall that (e.g. \cite{gabriel2012calculus, goerss2009simplicial}) simplicial anodyne extensions are those monomorphisms belonging to the saturation of the horn inclusions $\Lambda^m_k\to \Delta^m$ where $m\ge 1$ and $0\le k\le m$.

\begin{definition}[anodyne extension]
	A monomorphism of bisimplicial sets is an anodyne extension if it belongs to the saturated class of monomorphisms generated by the maps
\[
\Lambda^m_k\boxtimes\Delta^n\cup \Delta^m\boxtimes\Lambda^n_l\to \Delta^m\boxtimes\Delta^n
\]
for $m\ge 1$, $n\ge 1$, $0\le k\le m$ and $0\le l\le n$.
\end{definition}

\begin{definition}[fibration]
	A bisimplicial map $p: X\to Y$ is a fibration (resp. trivial fibration) if it has the right lifting property (RLP) against all anodyne extensions (resp. monomorphisms). A bisimplicial set $X$ is fibrant (resp. contractible) if $X\to *$ is a fibration.
\end{definition}
\begin{remark}
	\begin{enumerate}[(1)]
		\item It is clear that fibrations (resp. trivial fibrations) are stable under composition, base change (i.e. pull-back), and any isomorphism is a (trivial) fibration.
		\item By a standard small object argument, anodyne extensions are \textit{exactly} those monomorphisms having left lifting property (LLP) against all fibrations. See \cite[Theorem 2.1.14, Corollary 2.1.15]{hovey2007model}.
	\end{enumerate}
\end{remark}
\begin{definition}[Reedy matching objects]
	Let $X$ be a bisimplicial set and $K$ a simplicial set. We define simplicial sets $M_K(X)$ and $\overline{M}_K(X)$ to be
	\begin{align*}
		M_K(X)_m &=Hom_\ss(K, X_{m,\bullet}),\\
		\overline{M}_K(X)_n&=Hom_\ss(K, X_{\bullet,n}),
	\end{align*}
	with naturally induced simplicial structures. 
\end{definition}

\begin{proposition}[cf. \cite{moore1954homotopie}]\label{bisimplicial-Moore}
	The underlying bisimplicial set of a bisimplicial group is fibrant. More generally, epimorphisms between bisimplicial groups are fibrations.
\end{proposition}
Recall that Moore's theorem \cite{moore1954homotopie} asserts that every epimorphism of simplicial groups is a Kan fibration, and in particular every simplicial group is a Kan complex.

\begin{proof}
	Let $G$ be a bisimplicial group. Then by Moore's theorem, the natural restriction map
	\[
	i^*: M_{\Delta^n}(G)\to M_{\Lambda^{n}_l}(G)
	\]
	is an epimorphism of simplicial groups, and thus again by Moore's theorem is a Kan fibration. Therefore $i^*$ has RLP with respect to $\Lambda^m_k\to \Delta^m$, implying that $G$ is fibrant.
	The proof for the general case is the same, it suffices to replace $i^*$ by
	$M_{\Delta^n}(G)\to M_{\Lambda^{n}_l}(G)\times_{M_{\Lambda^n_l}(H)}M_{\Delta^n}(H)$ where $G\to H$ is an epimorphism.
\end{proof}

\begin{lemma}
Let $i: A\to B$, $j: K\to L$ be anodyne extensions of simplicial sets. Then
\[
i\hat{\boxtimes} j: A\boxtimes L\cup_{A\boxtimes K} B\boxtimes K\to B\boxtimes L
\]
is anodyne.
\end{lemma}
\begin{proof}
	Let $p: X\to Y$ be a fibration and $j$ a simplicial anodyne extension. To show that $i\hat{\boxtimes}j$ is anodyne amounts to showing that the dashed arrow in
	\[
	\begin{tikzcd}
		A\boxtimes L\cup_{A\boxtimes K} B\boxtimes K\ar[d,"i\hat{\boxtimes}j"']\ar[r] & X\ar[d,"p"]\\
		B\boxtimes L\ar[r]\ar[ur,dotted] & Y
	\end{tikzcd}
	\]
	exists for each $p$. But this is equivalent to that dashed arrow in
	\[
	\begin{tikzcd}
		A\ar[d,"i"']\ar[r] & M_L(X)\ar[d,"{(p_*,j^*)}"]\\
		B\ar[r]\ar[ur,dotted] & M_L(Y)\times M_K(X)
	\end{tikzcd}
	\]
	exists. Since the class of monomorphisms $i$ having LLP against $(p_*,j^*)$ is saturated, we can assume that $i$ is a generating simplicial anodyne extension, say $i:\Lambda^m_k\to \Delta^m$. By symmetry we can also assume that $j$ is a generating simplicial anodyne extension, say $j:\Lambda^n_l\to \Delta^n$. Then $i\hat{\boxtimes}j$ is by definition anodyne. 
\end{proof}

\begin{corollary}[Gabriel-Zismann property]
	The class of anodyne extensions is generated by either of the following classes of monomorphisms:
	\begin{enumerate}[(1)]
		\item $\Lambda^m_k\boxtimes\Delta^n\cup \Delta^m\boxtimes\Lambda^n_l\to \Delta^m\boxtimes\Delta^n$ for $m\ge 1$, $n\ge 1$, $0\le k\le m$ and $0\le l\le n$.
		\item $i\hat{\boxtimes}j$ where $i,j$ are simplicial anodyne extensions.
		\item $i\hat{\boxtimes}j$ where $i,j$ belong to a generating set of simplicial anodyne extensions.
		\item $\delta\Delta^{m,n}\times\Delta^{1,1}\cup \Delta^{m,n}\times \mathsf{L}_{\epsilon,\mu}^{1,1}\to \Delta^{m,n}\times\Delta^{1,1}$ where \[\mathsf{L}_{\epsilon,\mu}^{1,1}
		=\Lambda^1_\epsilon\boxtimes\Delta^1\cup\Delta^1\boxtimes \Lambda^1_\mu\]
		in which $\epsilon,\mu\in\{0,1\}$, and $m\ge 0, n\ge 0$.
		\item $A\times \Delta^{1,1}\cup B\times \mathsf{L}^{1,1}_{\epsilon,\mu}\to B\times \Delta^{1,1}$ for all monomorphisms $A\to B$ and $\epsilon,\mu\in\{0,1\}$.
	\end{enumerate}
\end{corollary}
\begin{proof}
	(1) is definition. (2) follows from the previous lemma and that maps in (1) are special cases of (2). (3) follows from (2) by the proof of the previous lemma. (4) follows from (3) by recalling that the class of simplicial anodyne extensions is also generated by
	\[
	\partial\Delta^m\times\Delta^1\cup \Delta^m\times \Lambda^1_\epsilon\to \Delta^m\times\Delta^1.
	\]
	(5) follows from (4) by noting that $B$ can be obtained from $A$ by attaching $\Delta^{m,n}$ along $\delta\Delta^{m,n}\to \Delta^{m,n}$.
\end{proof}

\begin{proposition}\label{pushout-product}
	Let $i: A\to B$ and $j: K\to L$ be monomorphisms of bisimplicial sets. Then the induced monomorphism
	\[
	A\times L \cup_{A\times K} B\times K\to  B\times L
	\]
	is anodyne if either one of $i$ or $j$ is anodyne.
\end{proposition}
\begin{proof}
	For a fixed $j$, the class of monomorphisms $i$ for which the map in question is anodyne is saturated, so we can take $i$ to be a map of the form (5) in the previous corollary. Then a straightforward computation shows that the map under inspection is also of the form (5).
	
	Alternatively for $i:\delta\Delta^{p,q}\to \Delta^{p,q}$ and $j: \Lambda^m_k\boxtimes\Delta^n\cup\Delta^m\boxtimes\Lambda^n_l\to \Delta^{m}\boxtimes\Delta^n$. The map under inspection is
	\begin{align*}
		(\Lambda^m_k\times\Delta^p \cup \Delta^m\times \partial\Delta^p)\boxtimes(\Delta^n\times\Delta^q)&\cup (\Delta^m\times\Delta^p)\boxtimes(\Lambda^n_l\times\Delta^q\cup \Delta^n\times\partial\Delta^q)\\
		&\to (\Delta^m\times\Delta^p)\boxtimes(\Delta^n\times\Delta^q).
	\end{align*}
	Since $\Lambda^m_k\times\Delta^p \cup \Delta^m\times \partial\Delta^p\to \Delta^m\times\Delta^p$ and $\Lambda^n_l\times\Delta^q\cup \Delta^n\times\partial\Delta^q\to \Delta^n\times\Delta^q$ are simplicial anodyne maps, so is the above map by the previous corollary.
\end{proof}
\begin{remark}
	The class of anodyne extensions is chosen so that \textit{both} \Cref{bisimplicial-Moore} and \Cref{pushout-product} hold. The author has tried other definitions of anodyne extensions, but usually either \Cref{bisimplicial-Moore} or \Cref{pushout-product} fails. For example, if one allows $\partial\Delta^m\boxtimes \Delta^n\cup\Delta^m\boxtimes\partial\Delta^n\to \Delta^{m,n}$ to be anodyne, then \Cref{bisimplicial-Moore} fails.
\end{remark}

The following examples and variants will be frequently used later.
\begin{example}\label{example-1}
	Let $\Lambda^{m+1}_{p,q}$ denote $\bigcup_{i\neq p,q}d_i\Delta^{m+1}$ and let $\Gamma^{m+1,n+1}$ denote \[\Lambda^{m+1}_{m,m+1}\boxtimes\Delta^{n+1}\cup\Delta^{m+1}\boxtimes\Lambda^{n+1}_{n,n+1}.\]
	Then
	\[
	\Gamma^{m+1,n+1}\cup d_m\dbar_n\Delta^{m+1,n+1}\hookrightarrow\Delta^{m+1,n+1}
	\]
	is anodyne for $m\ge 1$ and $n\ge 1$. This follows from that $\Lambda^{m+1,n+1}\hookrightarrow\Delta^{m+1,n+1}$ is anodyne and the pushout diagram below.
	\[
	\begin{tikzcd}
		\Lambda^{m,n+1}\sqcup \Lambda^{m+1,n}\ar[d]\ar[r]& \Gamma^{m+1,n+1}\cup d_m\dbar_n\Delta^{m+1,n+1}\ar[d]\\
		\Delta^m\boxtimes\Delta^{n+1}\sqcup \Delta^{m+1}\boxtimes\Delta^n\ar[r,"d_m\boxtimes 1\sqcup 1\boxtimes \dbar_n"] & \Lambda^{m+1,n+1}
	\end{tikzcd}
	\]
	Similarly $\Gamma^{m+1,n+1}\cup d_{m+1}\dbar_{n+1}\Delta^{m+1,n+1}\hookrightarrow\Delta^{m+1,n+1}$ is also anodyne.
\end{example}

\begin{example}\label{example-2} Let $\Lambda^{m+2}_{<m}$ denote $\bigcup_{i<m}d_i\Delta^{m+2}$. Then the inclusion
	\[
	(\Lambda^{m+2}_{<m}\boxtimes\Delta^{n+2}\cup\Delta^{m+2}\boxtimes\Lambda^{n+2}_{<n})\cup d_{m+1}\dbar_{n+1}\Delta^{m+2,n+2}\hookrightarrow\Delta^{m+2,n+2}
	\]
	is anodyne. To see this, we note that one can first attach
	\[
	d_{m+1}\Delta^{m+2,n+2}\cup \dbar_{n+1}\Delta^{m+2,n+2}\cong \Delta^{m+1}\boxtimes\Delta^{n+2}\cup \Delta^{m+2}\boxtimes\Delta^{n+1}
	\]
	to the domain along their intersection
	\[
	(\Lambda^{m+1}_{m,m+1}\boxtimes\Delta^{n+2}\cup \Delta^{m+1}\boxtimes\Lambda^{n+2}_{n,n+2})\cup (\Lambda^{m+2}_{m,m+2}\boxtimes\Delta^{n+1}\cup \Delta^{m+2}\boxtimes\Lambda^{n+1}_{n,n+1}).
	\]
	The resulting space is $\Lambda^{m+2}_{m,m+2}\boxtimes\Delta^{n+2}\cup \Delta^{m+2}\boxtimes\Lambda^{n+2}_{n,n+2}$ whose inclusion into $\Delta^{m+2,n+2}$ is anodyne.
\end{example}

\subsection{Mapping spaces}
\begin{definition}
	Let $X,Y$ be bisimplicial sets. Define the mapping space $\Fun(X,Y)$ to be the bisimplicial set
\[
\Fun(X,Y)_{m,n}=Hom_\bss(X\times\Delta^{m,n}, Y)
\]
with the naturally induced bisimplicial structure.
\end{definition}
\begin{remark}[exponential law]
	There is a natural evaluation map 
	\[
	X\times \Fun(X, Y)\xrightarrow{eval.} Y,
	\]
which induces a bijection
	\[
	Hom_{\bss}(K, \Fun(X,Y))\to Hom_\bss(K\times X, Y)
	\]
	that is natural in $K, X$ and $Y$.
\end{remark}

Let $i:A\to B$ be a monomorphism, and $p: X\to Y$ a fibration, consider the  general lifting problem
\[
\begin{tikzcd}
	A\ar[r]\ar[d,"i"'] & X\ar[d,"p"]\\
	B\ar[r]\ar[ur, dotted] & Y
\end{tikzcd}
\]

\begin{proposition} The induced map
	\[\Fun(B,X)\xrightarrow{(i^*, p_*)} \Fun(A,X)\times_{\Fun(A,Y)} \Fun(B,Y)\]
	is a fibration.
\end{proposition}
\begin{proof}
	This follows formally from \Cref{pushout-product}.
\end{proof}
\begin{corollary}\label{fibration-inherent}
	Let $i: L\to K$ be a monomorphism and $X$ a fibrant bisimplicial set. Then $i^*: \Fun(K,X)\to \Fun(L,X)$ is a fibration.
	
	In particular, if $X$ is fibrant then so is $\Fun(K,X)$ for all $K$.
	\end{corollary}
	\begin{proof}
		This follows immediately from the previous proposition.
	\end{proof}
	
\subsection{Path and loop spaces}
Recall that $\Lambda^{1,1}=\Lambda^1_1\boxtimes\Delta^1\cup \Delta^1\boxtimes\Lambda^{1}_1\subset \Delta^{1,1}$. We now use the pair $(\Delta^{1,1}, \Lambda^{1,1})$ to define path and loop functors.

\begin{definition}[path space]
	Let $(X,*)$ be a pointed bisimplicial set. We define the (based) path space $PX$ of $X$ by the following pull-back diagram
	\[
	\begin{tikzcd}
		PX\ar[r]\ar[d] & \Fun(\Delta^{1,1},X)\ar[d,"i^*"]\\
		*\ar[r] & \Fun(\Lambda^{1,1},X)
	\end{tikzcd}
	\]
	where $i:\Lambda^{1,1}\hookrightarrow\Delta^{1,1}$ is the inclusion and $*$ stands for the constant map. More explicitly, $PX$ is the bisimplicial set with
	\[
	(PX)_{m,n}=\{\alpha: \Delta^{m,n}\times\Delta^{1,1}\to X\ \text{such that}\ \alpha|_{\Delta^{m,n}\times\Lambda^{1,1}}=*\}.
	\]
\end{definition}
Note that by \Cref{fibration-inherent} if $X$ is fibrant, then $PX$ is fibrant. Moreover, we have
\begin{lemma}
	Let $(X,*)$ be a fibrant pointed bisimplicial set. Then $PX$ is contractible.
\end{lemma}
\begin{proof}
	It suffices to show that $i^*: \Fun(\Delta^{1,1},X)\to \Fun(\Lambda^{1,1},X)$ is a trivial fibration. Let $j: A\to B$ be a monomorphism, then the lifting problem
	\[
	\begin{tikzcd}
		A\ar[r]\ar[d,"j"] & \Fun(\Delta^{1,1},X)\ar[d,"i^*"]\\
		B\ar[r]\ar[ur,dotted] & \Fun(\Lambda^{1,1},X)
	\end{tikzcd}
	\]
	is equivalent to the lifting problem
	\[
	\begin{tikzcd}
		A\times\Delta^{1,1}\cup B\times\Lambda^{1,1}\ar[r]\ar[d] & X\\
		B\times\Delta^{1,1}\ar[ur,dotted] & 
	\end{tikzcd}.
	\]
	The dotted map exists because the vertical map is anodyne by \Cref{pushout-product}.
\end{proof}

\begin{definition}[loop space]
	Let $(X,*)$ be a pointed bisimplicial set. We define the (based) loop space $\Omega X$ of $X$ by the pull-back diagram
	\[
	\begin{tikzcd}
		\Omega X\ar[r]\ar[d] & \Fun(\Delta^{1,1},X)\ar[d,"i^*"]\\
		*\ar[r] & \Fun(\delta_\A\Delta^{1,1},X)
	\end{tikzcd}
	\]
	where $i:\delta_\A\Delta^{1,1}\to \Delta^{1,1}$ is the inclusion.
\end{definition}

Note that if $X$ is fibrant, then $\Omega X$ is fibrant. Moreover there is a pull-back diagram
\[
	\begin{tikzcd}
		\Omega X\ar[r]\ar[d] & PX\ar[d,"\mathrm{ev}"]\\
		*\ar[r] & X
	\end{tikzcd}
\]
in which the map $\mathrm{ev}: PX\to X$ is the fibration induced by evaluation at $d_1\dbar_1\Delta^{1,1}$.
\section{Homotopy groups}

\subsection{Definition of homotopy}
\begin{definition}[homotopy]\label{definition-homotopy}
	A homotopy between maps $\alpha, \beta: K\to X$ is a map $h: K\times \Delta^{1,1}\to X$ such that $d_0 h=\overline{s}_0\alpha$, $\overline{d}_0 h= s_0\alpha$ and $d_1\overline{d}_1 h=\beta$. We then say $\alpha$ is homotopic to $\beta$ through $h$, and denote $\alpha\sim_h \beta$.
\end{definition}

\begin{remark}
	Such a homotopy $h$ yields a $\Delta^{1,0}$-homotopy between $\alpha$ and $\beta$ in the sense that $\dbar_1 h$ is a map $K\times\Delta^{1,0}\to X$ satisfying that $d_0 (\dbar_1 h)=\alpha$ and $d_1(\dbar_1 h)=\beta$. Similarly $h$ also yields a $\Delta^{0,1}$-homotopy between $\alpha$ and $\beta$.
\end{remark}

\begin{definition}[relative homotopy]
	Let $X$ be a bisimplicial set, with a fixed vertex $*: \Delta^{0,0}\to X$. Let $(K,L)$ be a bisimplicial pair (i.e. $L\subset K$). A map $\alpha: K\to X$ rel $L$ is a map $\alpha: K\to X$ such that $\alpha|_{L}=*$.
	
	Let $\alpha,\beta$ be maps $K\to X$ rel $L$. A homotopy between $\alpha$ and $\beta$ relative to $L$ is a homotopy $h: K\times \Delta^{1,1}\to X$ between $\alpha$ and $\beta$ such that $h|_{L\times\Delta^{1,1}}=*$.
\end{definition}

The definition of homotopy appears to be asymmetric at first glance, however we have
\begin{proposition}
	Let $X$ be a fibrant bisimplicial set. Then homotopy defines an equivalence relation on vertices of $X$.
\end{proposition}
\begin{proof}
	Let $x$ be a vertex of $X$, then $s_0\overline{s}_0 x$ defines a homotopy $x\sim x$, thus proving reflexivity. Now suppose $x,x',x''$ are vertices of $X$ and given $x\sim_\alpha x'$ and $x\sim_\beta x''$, we show that $x'\sim x''$. Define a map
	\[
	\theta: \Gamma^{2,2}\cup d_1\dbar_1\Delta^{2,2}\to X
	\]
	by $d_0\theta=\overline{s}_0\alpha$, $\dbar_0\theta=s_0\alpha$ and $d_1\dbar_1\theta=\beta$. Note that $\theta$ is well-defined, for example $$x=d_0\dbar_0\beta=d_0\dbar_0(d_1\dbar_1\theta)=d_0\dbar_0\beta=d_0\dbar_0(d_0\dbar_0\theta)=d_0\dbar_0\alpha=x.$$
	Therefore (by \Cref{example-1}) $\theta$ extends to a map, $\Delta^{2,2}\to X$, still denoted by $\theta$. Now we claim that $d_2\dbar_2\theta$ defines a homotopy $x'\sim x''$. Indeed, we compute
	\[
		d_0(d_2\dbar_2\theta)=d_1\dbar_2 d_0\theta=d_1\dbar_2 \overline{s}_0\alpha=d_1\overline{s}_0\dbar_1\alpha=\overline{s}_0 x',
	\]
	and similarly $\dbar_0(d_2\dbar_2\theta)=s_0 x'$. Moreover we have that
	\[
	d_1\dbar_1(d_2\dbar_2\theta) =d_1\dbar_1(d_1\dbar_1\theta)=d_1\dbar_1\beta=x''.
	\]
	This completes the proof.
\end{proof}

\begin{corollary}
	Let $X$ be a fibrant bisimplicial set. Then homotopy defines an equivalence relation on maps $K\to X$.
	
	If  moreover $X$ is pointed, then relative homotopy defines an equivalence relation on maps $K\to X$ rel $L$.
\end{corollary}
\begin{proof}
	Apply the previous proposition to the fibrant bisimplicial sets $\Fun(K,X)$ and $\Fun(K,X)\times_{\Fun(L,X)}\Fun(L,*)$.
\end{proof}

There are clearly other variants of the definition of homotopy. For example, following the notations of \Cref{definition-homotopy}, we define an \textit{reversed homotopy} between maps $\alpha,\beta$ to be a map $g: K\times \Delta^{1,1}\to X$ such that $d_0\dbar_0 g=\alpha$, $d_1 g=\overline{s}_0\beta$ and $\dbar_1 g= s_0\beta$. We then say $\alpha$ is reversed homotopic to $\beta$ through $g$. With these understood, we have
\begin{lemma}
	If $X$ is fibrant then $\alpha$ is homotopic to $\beta$ if and only if $\alpha$ is reversed homotopic to $\beta$.
\end{lemma}
\begin{proof}
	We prove one implication, the other is similar. Suppose that $\alpha$ is homotopic to $\beta$ through $h$, then we can define a map
	\[
	\theta: \Lambda^2_2\boxtimes\Delta^2\cup \Delta^2\boxtimes\Lambda^2_2\to X
	\]
	by $d_0\theta=\overline{s}_1 h$, $\dbar_0\theta= s_1 h$, $d_1\theta=\overline{s}_0 h$ and $\dbar_1\theta=s_0 h$. Now $\theta$ is well-defined and extends to a map $\theta: \Delta^{2,2}\to X$. One readily checks that $d_2\dbar_2\theta$ gives a reversed homotopy between $\alpha$ and $\beta$.
\end{proof}

\subsection{Homotopy sets}
\begin{definition}[Bott-Chern and Aeppli homotopy]
	The Bott-Chern homotopy set of a pointed fibrant bisimplicial set $(X,*)$ in degree $(m,n)$, denoted by $\pi_{m,n}^\BC(X,*)$, is the set of relative homotopy classes of maps
	\[
	\Delta^{m,n}\to X \ \text{rel }\delta\Delta^{m,n}.
	\]
	
	The Aeppli homotopy set in degree $(m,n)$, denoted by $\pi_{m,n}^\A(X,*)$, is the set of relative homotopy classes of maps
	\[
	\Delta^{m,n}\to X \ \text{rel }\delta_\A\Delta^{m,n}.
	\]
	
	For simplicity, we may sometimes suppress the base point and write $\pi_{m,n}^\BC(X)$, $\pi_{m,n}^\A(X)$ for $\pi_{m,n}^\BC(X,*)$, $\pi_{m,n}^\A(X,*)$ respectively.
\end{definition}

It is clear that a map between pointed fibrant bisimplicial sets induces, by composition, a map between Bott-Chern and Aeppli homotopy sets. Moreover, homotopic maps that preserve basepoint induce the same map on Bott-Chern and Aeppli homotopy sets.

\begin{lemma}\label{character-BC}
	A map $\alpha: \Delta^{m,n}\to X$ rel $\delta\Delta^{m,n}$ is homotopic to the constant map $*$ rel $\delta\Delta^{m,n}$ if and only if there exists a map $\theta: \Delta^{m+1,n+1}\to X$ such that
	$d_i\theta=*$, $0\le i\le m$; $\dbar_j\theta=*$, $0\le j\le n$ and $d_{m+1}\dbar_{n+1}\theta=\alpha$.
\end{lemma}
 Recall the notations $\Lambda^{m,n}=\Lambda^{m}_{m}\boxtimes\Delta^{n}\cup\Delta^{m}\boxtimes\Lambda^{n}_{n}$ and $\delta_\A\Delta^{m,n}=\Lambda^{m,n}\cup d_m\dbar_n\Delta^{m,n}$.

\begin{proof}
	Suppose that $*\sim_h \alpha$ rel $\delta\Delta^{m,n}$ then we define a map
	\[
	\omega: \delta_\A\Delta^{m+1,n+1}\times\Delta^{1,1}\cup \Delta^{m+1,n+1}\times \Lambda^{1,1}\to X
	\]
	by $\omega=*$ on $\Lambda^{m+1,n+1}\times\Delta^{1,1}$; $\omega=*$ on $\Delta^{m+1,n+1}\times \Lambda^{1,1}$ and $\omega=h$ on $d_{m+1}\dbar_{n+1}\Delta^{m+1,n+1}\times \Delta^{1,1}$. Then one checks that $\omega$ is well-defined and thus (by \Cref{pushout-product}) extends to a map $\omega:\Delta^{m+1,n+1}\times\Delta^{1,1}\to X$. Let $\theta$ be the restriction of $\omega$ onto $\Delta^{m+1,n+1}\times d_1\dbar_1\Delta^{1,1}$, then $\theta$ is a required map.
	
	Conversely, given such a map $\theta$, we define a map
	\[
	\omega: \Lambda^{m+1,n+1}\times\Delta^{1,1}\cup \Delta^{m+1,n+1}\times \delta_\A\Delta^{1,1}\to X
	\]
	by $\omega=*$ on $\Lambda^{m+1,n+1}\times\Delta^{1,1}$; $\omega=*$ on $\Delta^{m+1,n+1}\times \Lambda^{1,1}$ and $\omega=\theta$ on $\Delta^{m+1,n+1}\times d_1\dbar_1\Delta^{1,1}$. Then $\omega$ extends to a map $\omega:\Delta^{m+1,n+1}\times\Delta^{1,1}\to X$ whose restriction onto $d_{m+1}\dbar_{n+1}\Delta^{m+1,n+1}\times\Delta^{1,1}$ gives a homotopy $*\sim \alpha$ rel $\delta\Delta^{m,n}$.
\end{proof}

\begin{lemma}\label{character-A}
	A map $\alpha: \Delta^{m,n}\to X$ rel $\delta_\A\Delta^{m,n}$ is homotopic to the constant map $*$ rel $\delta_\A\Delta^{m,n}$ if and only if there exists a map $\theta: \Delta^{m+1,n+1}\to X$ such that
	$d_i\theta=*$, $0\le i< m$; $\dbar_j\theta=*$, $0\le j< n$; $d_m\dbar_n\theta=*$, and $d_{m+1}\dbar_{n+1}\theta=\alpha$.
\end{lemma}
\begin{proof}
	Replace $\Lambda^{m+1,n+1}$ by $\Gamma^{m+1,n+1}\cup d_m\dbar_n\Delta^{m+1,n+1}$ in the previous proof.
\end{proof}

\begin{corollary}
	If $X$ is contractible, then all of its Bott-Chern and Aeppli homotopy sets are singletons.
\end{corollary}
\begin{proof}
	The required maps $\theta$ in the previous lemmas always exist.
\end{proof}

\subsection{The Bott-Chern product}\label{BC-group-operation}
In this and the following subsection, we fix a pointed fibrant bisimplicial set $(X,*)$ and assume that $m\ge 1$, $n\ge 1$.

Given $\alpha,\beta:\Delta^{m,n}\to X$ rel $\delta\Delta^{m,n}$, we define a map
\[
\theta: \Lambda^{m+1}_m\boxtimes\Delta^{n+1}\cup \Delta^{m+1}\boxtimes \Lambda^{n+1}_n\to X
\]
by $d_i\theta=*$, $0\le i<m-1$; $\dbar_j\theta=*$, $0\le j<n-1$; $d_{m-1}\theta=\overline{s}_{n-1}\alpha$, $\dbar_{n-1}\theta=s_{m-1}\alpha$; $d_{m+1}\theta=\overline{s}_n\beta$, $\dbar_{n+1}\theta=s_m\beta$.

Then one readily checks that $\theta$ is well-defined and therefore extends to a map $\theta: \Delta^{m+1,n+1}\to X$. Moreover, the map $d_m\dbar_n\theta:\Delta^{m,n}\to X$ is constant on $\delta\Delta^{m,n}$ and thus defines a class in $\pi_{m,n}^\BC(X)$.

\begin{lemma}\label{BC-product-well-define}
	The class $[d_m\dbar_n\theta]\in \pi_{m,n}^\BC(X)$ depends only on the classes $[\alpha],[\beta]\in \pi_{m,n}^\BC(X)$.
\end{lemma}
\begin{proof}
	Suppose that $\alpha\sim_h\alpha'$ and $\beta\sim_g\beta'$ rel $\delta\Delta^{m,n}$ and let $\theta$ (resp. $\theta'$) be constructed as above for $\alpha$ and $\beta$ (resp. $\alpha'$ and $\beta'$). We define a map
	\[
	\omega: (\Lambda^{m+1}_m\boxtimes\Delta^{n+1}\cup\Delta^{m+1}\boxtimes\Lambda^{n+1}_n)\times\Delta^{1,1}\cup \Delta^{m+1,n+1}\times \delta_\A\Delta^{1,1}\to X
	\]
	by $d_i\omega=*$, $i<m-1$; $\dbar_j\omega=*$, $j<n-1$; $d_{m-1}\omega=\overline{s}_{n-1}h$, $\dbar_{n-1}\omega=s_{m-1}h$; $d_{m+1}\omega=\overline{s}_n g$, $\dbar_{n+1}\omega=s_m g$; and moreover $\hat{d}_0\omega=\hat{\overline{s}}_0\theta$, $\hat{\dbar}_0\omega=\hat{s}_0\theta$ and $\hat{d}_1\hat{\dbar}_1\omega=\theta'$. Here $\{d_i, s_i; \dbar_j,\overline{s}_j\}$ are the bisimplicial structure maps in $\Delta^{m+1,n+1}$-direction, while $\{\hat{d}_i, \hat{s}_i; \hat{\dbar}_j,\hat{\overline{s}}_j\}$ are the bisimplicial structure maps in $\Delta^{1,1}$-direction.
	
	Then one checks that $\omega$ is well-defined and thus (by \Cref{pushout-product}) extends to a map $\omega: \Delta^{m+1,n+1}\times\Delta^{1,1}\to X$. The map $d_m\dbar_n\omega$ defines a relative homotopy between $d_m\dbar_n\theta\sim d_m\dbar_n\theta'$.
\end{proof}

\begin{remark}\label{interval-direction}
	We will frequently encounter bisimplicial sets of the form $X\times \Delta^{1,1}$ and its sub-bisimplicial sets. Just as in the proof above, we shall write $\{\hat{d}_i, \hat{s}_i; \hat{\dbar}_j,\hat{\overline{s}}_j\}$ for the bisimplicial structure maps in $\Delta^{1,1}$-direction.
\end{remark}

Therefore, we define the Bott-Chern product of $[\alpha],[\beta]$ to be
\[
[\alpha]\star[\beta]:=[d_m\dbar_n\theta].
\]

\begin{lemma}\label{BC-product-unit}
	The class of the constant map $[*]$ is a two-sided unit for $\star$.
\end{lemma}
\begin{proof}
	Given $\alpha: \Delta^{m,n}\to X$ rel $\delta\Delta^{m,n}$. Take $\theta=s_{m-1}\overline{s}_{n-1}\alpha$, then by definition
	\[
	[\alpha]\star[*]=[d_m\dbar_n\theta]=[\alpha].
	\]
	Similarly, take $\theta=s_m\overline{s}_n\alpha$, then by definition
	\[
	[*]\star[\alpha]=[d_m\dbar_n\theta]=[\alpha].
	\]
\end{proof}

\begin{lemma}\label{BC-product-associative}
	The Bott-Chern product $\star$ is associative.
\end{lemma}
\begin{proof}
	Given $\alpha,\beta,\gamma: \Delta^{m,n}\to X$ rel $\delta\Delta^{m,n}$. We can follow the definition to construct $\theta$, $\theta'$ and $\theta''$ such that
	$\sigma=d_m\dbar_n\theta$ represents $[\alpha]\star[\beta]$, $d_m\dbar_n\theta'$ represents $[\sigma]\star[\gamma]$, and $\eta=d_m\dbar_n\theta''$ represents $[\beta]\star[\gamma]$. Using these, we can define a map
	\[
	\omega: (\Lambda^{m+2}_{m,m+1}\boxtimes\Delta^{n+2}\cup\Delta^{m+2}\boxtimes\Lambda^{n+2}_{n,n+1})\cup d_{m+1}\dbar_{n+1}\Delta^{m+2,n+2}\to X
	\]
	by $d_i\omega=*$, $i<m-1$; $\dbar_j\omega=*$, $j<n-1$; $d_{m-1}\omega=\overline{s}_{n-1}\theta$, $\dbar_{n-1}\omega=s_{m-1}\theta$; $d_{m+1}\dbar_{n+1}\omega=\theta'$; and $d_{m+2}\omega=\overline{s}_{n+1}\theta''$, $\dbar_{n+2}\omega=s_{m+1}\theta''$.
	
	Then one can check that $\omega$ is well-defined, for example
	\begin{align*}
		d_{m-1}(d_{m+1}\dbar_{n+1}\omega)& =d_{m-1}\theta'=\overline{s}_{n-1}\sigma\\
		&= \overline{s}_{n-1} (d_m\dbar_n\theta)=d_m\dbar_{n+1}\overline{s}_{n-1}\theta =d_m \dbar_{n+1}(d_{m-1}\omega).
	\end{align*}

	It follows (by \Cref{example-1}) that $\omega$ extends to a map $\omega:\Delta^{m+2,n+2}\to X$. Now let us examine $\xi=d_m\dbar_n\omega$. It is clear that $d_i\xi=*$, $i<m-1$ and $\dbar_j\xi=*$, $j<n-1$, and we compute
	\begin{align*}
		d_{m-1}\xi &=d_{m-1}(d_m\dbar_n\omega)=d_{m-1}\dbar_n(d_{m-1}\omega)=d_{m-1}\dbar_n\overline{s}_{n-1}\theta\\
		&=d_{m-1}\theta=\overline{s}_{m-1}\alpha,\\
		d_{m+1}\xi &=d_{m+1}(d_m \dbar_n\omega)=d_m\dbar_n (d_{m+2}\omega)=d_m\dbar_n\overline{s}_{n+1}\theta''\\
		&=\overline{s}_n (d_m\dbar_n\theta'')=\overline{s}_n\eta.
	\end{align*}
	Similarly we have $\dbar_{n-1}\xi=s_{m-1}\alpha$ and $\dbar_{n+1}\xi=s_m\eta$. This implies that
	\begin{align*}
		[\alpha]\star([\beta]\star[\gamma]) &= [\alpha]\star[\eta]=[d_m\dbar_n\xi]\\
		&=[d_m \dbar_n (d_m\dbar_n\omega)]=[d_m \dbar_n (d_{m+1}\dbar_{n+1}\omega)]\\
		&=[d_m\dbar_n\theta']=[\sigma]\star[\gamma]\\
		&=([\alpha]\star[\beta])\star[\gamma].
	\end{align*}
	This proves that $\star$ is associative.
\end{proof}

\begin{theorem} Let $(X,*)$ be a pointed fibrant bisimplicial set.
	For $m\ge 1, n\ge 1$, the Bott-Chern product $\star$ turns $\pi_{m,n}^\BC(X,*)$ into a monoid.
\end{theorem}
\begin{proof}
	This follows from the previous lemmas.
\end{proof}

\subsection{The Aeppli product}
Recall that we fix a pointed fibrant bisimplicial set $(X,*)$ and assume that $m\ge 1$, $n\ge 1$. Also recall the notation $\Gamma^{m+1,n+1}=\Lambda^{m+1}_{m,m+1}\boxtimes\Delta^{n+1}\cup\Delta^{m+1}\boxtimes\Lambda^{n+1}_{n,n+1}$.

Given $\alpha,\beta:\Delta^{m,n}\to X$ rel $\delta_\A\Delta^{m,n}$, we define a map
\[
\theta: \Gamma^{m+1,n+1}\cup d_{m+1}\dbar_{n+1}\Delta^{m+1,n+1}\to X
\]
by $d_i\theta=*$, $0\le i<m-1$; $\dbar_j\theta=*$, $0\le j<n-1$; $d_{m-1}\theta=\overline{s}_{n-1}\alpha$, $\dbar_{n-1}\theta=s_{m-1}\alpha$; and $d_{m+1}\dbar_{n+1}\theta=\beta$.

Then $\theta$ is well-defined and (by \Cref{example-1}) extends to a map $\theta:\Delta^{m+1,n+1}\to X$. Moreover the map $d_m\dbar_n\theta$ is constant on $\delta_\A\Delta^{m,n}$.
\begin{lemma}\label{A-product-well-define}
	The class $[d_m\dbar_n\theta]\in \pi_{m,n}^\A(X)$ depends only on the classes $[\alpha], [\beta]\in \pi_{m,n}^\A(X)$.
\end{lemma}
\begin{proof}
	The proof is similar to \Cref{BC-product-well-define}, it suffices to replace $\Lambda^{m+1}_m\boxtimes\Delta^{n+1}\cup\Delta^{m+1}\boxtimes\Lambda^{n+1}_n$ by $\Gamma^{m+1,n+1}\cup d_{m+1}\dbar_{n+1}\Delta^{m+1,n+1}$.
\end{proof}

Therefore we define the Aeppli product of $[\alpha],[\beta]$ to be
\[
[\alpha]\diamond [\beta]:=[d_m\dbar_n\theta].
\]
\begin{lemma}
	The class of the constant map $[*]$ is a two-sided unit for $\diamond$.
\end{lemma}
\begin{proof}
	The proof is identical to \Cref{BC-product-unit}.
\end{proof}

\begin{lemma}
	The Aeppli product $\diamond$ is associative.
\end{lemma}
\begin{proof}
	Given $\alpha,\beta,\gamma: \Delta^{m,n}\to X$ rel $\delta_\A\Delta^{m,n}$. Then following the definition we get $\theta$ with $\sigma=d_m\dbar_n\theta$ representing $[\alpha]\diamond[\beta]$, and $\theta'$ with $d_m\dbar_n\theta'$ representing $[\sigma]\diamond[\gamma]$. Using these we define a map
	\[
	\omega: (\Lambda^{m+2}_{<m}\boxtimes\Delta^{n+2}\cup\Delta^{m+2}\boxtimes\Lambda^{n+2}_{<n})\cup d_{m+1}\dbar_{n+1}\Delta^{m+2,n+2}\to X
	\]
	by $d_i\omega=*$, $i<m-1$; $\dbar_j\omega=*$, $j<n-1$; $d_{m-1}\omega=\overline{s}_{n-1}\theta$, $\dbar_{n-1}\omega=s_{m-1}\theta$ and $d_{m+1}\dbar_{n+1}\omega=\theta'$. One checks that $\omega$ is well-defined and (by \Cref{example-2}) extends to a map $\omega: \Delta^{m+2,n+2}\to X$.
	
	A direct computation shows that $\theta''=d_{m+2}\dbar_{n+2}\omega$ is a map that by definition defines $[\beta]\diamond[\gamma]$ such that $\eta=d_m\dbar_n \theta''$ represents $[\beta]\diamond[\gamma]$. Moreover one checks that $\xi=d_m\dbar_n\omega$ by definition defines $[\alpha]\diamond [\eta]$. The rest is the same as the proof of \Cref{BC-product-associative}.
\end{proof}

\begin{lemma}
	The Aeppli product $\diamond$ is right divisible. In particular $\diamond$ is right invertible.
\end{lemma}
\begin{proof}
	Given $\alpha,\beta:\Delta^{m,n}\to X$ rel $\delta_\A\Delta^{m,n}$. We define a map
	\[
	\theta: \Gamma^{m+1,n+1}\cup d_m\dbar_n\Delta^{m+1,n+1}\to X
	\] by $d_i\theta=*$, $0\le i<m-1$; $\dbar_j\theta=*$, $0\le j<n-1$; $d_{m-1}\theta=\overline{s}_{n-1}\alpha$, $\dbar_{n-1}=s_{m-1}\alpha$; and $d_{m}\dbar_{n}\theta=\beta$. Then $\theta$ is well-defined and (by \Cref{example-1}) extends to a map $\theta: \Delta^{m+1,n+1}\to X$. By definition we have
	\[
	[\alpha]\diamond [d_{m+1}\dbar_{n+1}\theta]=[\beta]
	\]
	proving that $\diamond$ is right divisible.
\end{proof}

\begin{theorem} Let $(X,*)$ be a pointed fibrant bisimplicial set.
	For $m\ge 1, n\ge 1$, the Aeppli product $\diamond$ turns $\pi_{m,n}^\A(X,*)$ into a group.
\end{theorem}
\begin{proof}
	This follows from the previous lemmas. We note that a monoid whose multiplication is right invertible is in fact a group.
\end{proof}

\subsection{Relation between Aeppli and Bott-Chern}
The goal of this subsection is to establish that $\pi_{m,n}^\BC(\Omega X)\cong \pi^\A_{m+1,n+1}(X)$. As usual we fix a pointed fibrant bisimplicial set $(X,*)$.

To start with, we construct a map $\rho: \pi^\A_{m+1,n+1}(X)\to \pi_{m,n}^\BC(\Omega X)$ as follows. Given $\alpha: \Delta^{m+1,n+1}\to X$ rel $\delta_\A\Delta^{m+1,n+1}$, consider the following lifting problem
\[
\begin{tikzcd}
	\Lambda^{m+1,n+1}\ar[r,"*"]\ar[d] & PX\ar[d,"\mathrm{ev}"]\\
	\Delta^{m+1,n+1}\ar[r,"\alpha"]\ar[ur,"\widetilde{\alpha}", dotted] & X
\end{tikzcd}
\]
where $*: \Lambda^{m+1,n+1}\to PX$ is the constant map. Then since $\mathrm{ev}: PX\to X$ is a fibration and $\Lambda^{m+1,n+1}\hookrightarrow\Delta^{m+1,n+1}$ is anodyne, the dotted map $\widetilde{\alpha}$ exists. 

It is straightforward to verify that $\rho_\alpha:=d_{m+1}\dbar_{n+1}\widetilde{\alpha} $ maps into $\Omega X$, and that $\rho_\alpha|_{\delta\Delta^{m,n}}$ is constant. Therefore $\rho_\alpha$ defines a class $[\rho_\alpha]\in \pi_{m,n}^\BC(\Omega X)$.

\begin{lemma}
	The class $[\rho_\alpha]\in \pi_{m,n}^\BC(\Omega X)$ depends only on the class $[\alpha]\in \pi_{m+1,n+1}^\A(X)$. 
\end{lemma}
\begin{proof}
	The proof is similar to \Cref{BC-product-well-define} and \Cref{A-product-well-define}. One shows that homotopy can be lifted along the anodyne extension $\Lambda^{m+1,n+1}\times\Delta^{1,1}\cup \Delta^{m+1,n+1}\times\delta_\A\Delta^{1,1}\to \Delta^{m+1,n+1}\times\Delta^{1,1}$ (recall \Cref{pushout-product}). We leave the details to the reader.
\end{proof}

Therefore we can define a map
\[
\rho: \pi^\A_{m+1,n+1}(X)\to \pi_{m,n}^\BC(\Omega X), \quad [\alpha]\mapsto [\rho_\alpha].
\]
Similarly we can construct a map $\mu: \pi_{m,n}^\BC(\Omega X)\to \pi^\A_{m+1,n+1}(X)$ as follows. Given $\alpha: \Delta^{m,n}\to \Omega X$ rel $\delta\Delta^{m,n}$, consider $\iota\alpha$ where $\iota: \Omega X\to PX$ is the canonical embedding. Notice that $PX$ is contractible, therefore there exists a map $\theta: \Delta^{m+1,n+1}\to PX$ such that $d_i\theta=*$, $0\le i\le m$; $\dbar_j\theta=*$, $0\le j\le n$; and $d_{m+1}\dbar_{n+1}\theta=\iota\alpha$.

It is now straightforward to verify that $\mu_\alpha=\mathrm{ev}\circ \theta: \Delta^{m+1,n+1}\to X$ is constant on $\delta_\A\Delta^{m+1,n+1}$. Therefore $\mu_\alpha$ defines a class $[\mu_\alpha]\in \pi_{m+1,n+1}^\A(X)$. 
\begin{lemma}
	The class $[\mu_\alpha]\in \pi_{m+1,n+1}^\A(X)$ depends only on the class $[\alpha]\in \pi_{m,n}^\BC(\Omega X)$. 
\end{lemma}
\begin{proof}
	We sketch the proof. Given a homotopy $\alpha\sim_h\beta$, we can find $\theta_\alpha$ and $\theta_\beta$ as above and then construct a map
	\[
	\omega: \Delta^{m+1,n+1}\times\Delta^{1,1}\to PX
	\]
	such that $d_i\omega=\dbar_j\omega=*$ for $i\le m, j\le n$ and that $d_{m+1}\dbar_{n+1}\omega=\iota h$ and $\hat{d}_0\omega=\hat{\overline{s}}_0\theta_\alpha$, $\hat{\dbar}_0\omega=\hat{s}_0\theta_\alpha$, $\hat{d}_1\hat{\dbar}_1\omega=\theta_\beta$. Then $\mathrm{ev}\circ\omega$ gives the desired homotopy.
\end{proof}

\begin{proposition} Let $(X,*)$ be a pointed fibrant bisimplicial set.
	For $m\ge 0$ and $n\ge 0$, the map $\rho: \pi^\A_{m+1,n+1}(X)\to \pi_{m,n}^\BC(\Omega X)$ is a bijection with inverse $\mu$.
\end{proposition}
\begin{proof}
	Given $\alpha: \Delta^{m+1,n+1}\to X$ rel $\delta_\A\Delta^{m+1,n+1}$, we observe that $\widetilde{\alpha}$ satisfies that $d_i\widetilde{\alpha}=*$, $0\le i\le m$; $\dbar_j\widetilde{\alpha}=*$, $0\le j\le n$; and $d_{m+1}\dbar_{n+1}\widetilde{\alpha}=\iota\rho_\alpha$. Therefore by definition we have $\mu\circ\rho[\alpha]=[\mathrm{ev}\circ \widetilde{\alpha}]=[\alpha]$. This proves that $\mu\circ \rho=\mathrm{id}$. Similarly one can prove that $\rho\circ \mu=\mathrm{id}$.
\end{proof}

Next we analyze the monoid products. The first observation is that there is another product on $\pi_{m,n}^\BC(\Omega X)$ induced by``composing loops". Indeed maps $\Delta^{m,n}\to \Omega X$ rel $\delta\Delta^{m,n}$ can be viewed as maps
\[
\Delta^{1,1}\to \Fun(\Delta^{m,n},X)\times_{\Fun(\delta\Delta^{m,n}, X)}\Fun(\delta\Delta^{m,n}, *)\ \text{rel}\ \delta_\A\Delta^{1,1}.
\]
Therefore one can follow the construction of Aeppli product to define an operation $\diamond$ on $\pi_{m,n}^\BC(\Omega X)$.
\begin{lemma}
	For $m\ge 0$, $n\ge 0$, the composition of loops operation $\diamond$ turns $\pi_{m,n}^\BC(\Omega X)$ into a group.
\end{lemma}
\begin{proof}
	This is the same as the proof of that the Aeppli product is a group operation.
\end{proof}

\begin{proposition}
	For $m\ge 1$, $n\ge 1$, the operations $\star$ and $\diamond$ on $\pi_{m,n}^\BC(\Omega X)$ coincide, and are abelian.
\end{proposition}
\begin{proof}
	This follows from a standard Eckmann-Hilton argument, we sketch the proof. Given $\alpha_1,\alpha_2, \beta_1,\beta_2: \Delta^{m,n}\times \Delta^{1,1}\to X$ rel $\delta\Delta^{m,n}\times\Delta^{1,1}\cup\Delta^{m,n}\times\delta_\A\Delta^{1,1}$, we show that
	\[
	([\alpha_1]\star[\alpha_2])\diamond([\beta_1]\star[\beta_2])=([\alpha_1]\diamond [\beta_1])\star([\alpha_2]\diamond[\beta_2]).
	\] 
	
	On the one hand, we can follow the definition of Bott-Chern product to obtain maps $\theta_\alpha, \theta_\beta: \Delta^{m+1,n+1}\times \Delta^{1,1}\to X$ rel $\Delta^{m+1,n+1}\times\delta_\A\Delta^{1,1}$ such that $\theta_\alpha$ defines $[\alpha_1]\star[\alpha_2]$ in the sense that $d_m\dbar_n\theta_\alpha$ represents $[\alpha_1]\star[\alpha_2]$; and similarly $\theta_\beta$ defines $[\beta_1]\star[\beta_2]$.
	
	On the other hand, we can follow the definition of Aeppli product to obtain maps $\sigma_1,\sigma_2: \Delta^{m,n}\times\Delta^{2,2}\to X$ rel $\delta\Delta^{m,n}\times\Delta^{2,2}$ such that $\sigma_1,\sigma_2$ define $[\alpha_1]\diamond [\beta_1]$ and $[\alpha_2]\diamond [\beta_2]$ respectively.
	
	Now we can construct a map $\omega: \Delta^{m+1,n+1}\times\Delta^{2,2}\to X$ such that
	\begin{enumerate}[(1)]
		\item $d_i\omega=*$, $i<m-1$; $\dbar_j\omega=*$, $j<n-1$; $d_{m-1}\omega=\overline{s}_{n-1}\sigma_1$, $\dbar_{n-1}\omega=s_{m-1}\sigma_1$; $d_{m+1}\omega=\overline{s}_n\sigma_2$, $\dbar_{n+1}\omega=s_m\sigma_2$.
		\item $\hat{d}_0\omega=\hat{\overline{s}}_0\theta_\alpha$, $\hat{\dbar}_0\omega=\hat{s}_0\theta_\alpha$, and $\hat{d}_1\hat{\dbar}_1\omega=\theta_\beta$.
	\end{enumerate}
	Indeed $\omega$ exists because the above equations partially define $\omega$ on
	\[
	(\Lambda^{m+1}_m\boxtimes\Delta^{n+1}\cup\Delta^{m+1}\boxtimes\Lambda^{n+1}_n)\times\Delta^{2,2}\cup \Delta^{m+1,n+1}\times (\Gamma^{2,2}\cup d_2\dbar_2\Delta^{2,2})
	\]
	whose inclusion into $\Delta^{m+1,n+1}\times\Delta^{2,2}$ is anodyne by \Cref{example-1} and \Cref{pushout-product}.
	
	Finally one checks that $d_m\dbar_n\hat{d}_1\hat{\dbar}_1\omega$ simultaneously represents both $([\alpha_1]\star[\alpha_2])\diamond([\beta_1]\star[\beta_2])$ and $([\alpha_1]\diamond [\beta_1])\star([\alpha_2]\diamond[\beta_2])$.
\end{proof}
\begin{corollary}
	Let $(X,*)$ be a pointed fibrant bisimplicial set. Then $\pi_{m,n}^\BC(\Omega X,*)$ is a group for $m\ge 0$, $n\ge 0$, and is abelian for $m\ge 1$, $n\ge 1$.
\end{corollary}
\begin{proof}
	This follows from the previous proposition and lemma.
\end{proof}

\begin{theorem}\label{Aeppli-is-BC-of-loop}
	Let $(X,*)$ be a pointed fibrant bisimplicial set. For $m\ge 0$, $n\ge 0$, the map
\[
	\rho: \left(\pi^\A_{m+1,n+1}(X), \diamond\right)\to \left(\pi_{m,n}^\BC(\Omega X), \diamond\right)
\]
	is an isomorphism of groups.
\end{theorem}
\begin{proof}
	Given $\alpha,\beta:\Delta^{m+1,n+1}\to X$ rel $\delta_\A\Delta^{m+1,n+1}$, we follow the definition of $\rho$ to construct maps $\widetilde{\alpha}, \widetilde{\beta}$, $\rho_\alpha=d_{m+1}\dbar_{n+1}\widetilde{\alpha}$ and $\rho_\beta=d_{m+1}\dbar_{n+1}\widetilde{\beta}$. 
	
	Now we define a map
	\[
	\nu: \Lambda^{m+1,n+1}\times\Delta^{2,2}\cup \Delta^{m+1,n+1}\times (\Gamma^{2,2}\cup \hat{d}_2\hat{\dbar}_2\Delta^{2,2})\to X
	\]
	by that
	\begin{enumerate}[(1)]
		\item $d_i\nu=*$, $i<m$; $\dbar_j\nu=*$, $j<n$; $d_m\nu=\hat{s}_0\hat{\overline{s}}_0\overline{s}_n\rho_\alpha$, $\dbar_n\nu=\hat{s}_0\hat{\overline{s}}_0 s_m\rho_\alpha$; and
		\item $\hat{d}_0\nu=\hat{\overline{s}}_0 s_m\overline{s}_n\rho_\alpha$, $\hat{\overline{d}}_0\nu=\hat{s}_0 s_m\overline{s}_n\rho_\alpha$, $\hat{d}_2\hat{\dbar}_2\nu=\widetilde{\beta}$.
	\end{enumerate}
	Then one checks that $\nu$ is well-defined and thus (by \Cref{example-1}, \Cref{pushout-product}) extends to a map $\nu:\Delta^{m+1,n+1}\times\Delta^{2,2}\to X$. It is straightforward to verify the following. 
	\begin{enumerate}[(1)]
		\item The map $\theta=d_{m+1}\dbar_{n+1}\nu$ satisfies that $\hat{d}_0\theta=\hat{\overline{s}}_0\rho_\alpha$, $\hat{\dbar}_0\theta=\hat{s}_0\rho_\alpha$ and $\hat{d}_2\hat{\dbar}_2\theta=\rho_\beta$. Therefore $\hat{d}_1\hat{\dbar}_1\theta$ represents $[\rho_\alpha]\diamond[\rho_\beta]$.
		\item The map $\omega=\hat{d}_1\hat{\dbar}_1\nu$ satisfies that $d_{i}\omega=*$, $i<m$; $\dbar_j\omega=*$, $j<n$; $d_m\omega=\overline{s}_n\rho_\alpha$, $\dbar_n\omega=s_m\rho_\alpha$; $d_{m+1}\dbar_{n+1}\omega=\hat{d}_1\hat{\dbar}_1\theta$ and moreover $\mathrm{ev}\circ\omega=\beta$. Note that by definition
		\begin{align*}
			\mathrm{ev}\circ\omega=\hat{d}_1\hat{\dbar}_1\omega=\hat{d}_1\hat{\dbar}_1(\hat{d}_1\hat{\dbar}_1\nu)=\hat{d}_1\hat{\dbar}_1(\hat{d}_2\hat{\dbar}_2\nu)=\hat{d}_1\hat{\dbar}_1\widetilde{\beta}=\mathrm{ev}\circ\widetilde{\beta}=\beta.
		\end{align*}
	\end{enumerate}
	
	Next we can define a map
	\[
	H:\Gamma^{m+2,n+2}\cup d_{m+2}\dbar_{n+2}\Delta^{m+2,n+2}\to PX
	\]
	by $d_i H=*$, $i<m$; $\dbar_j H=*$, $j<n$; $d_m H=\overline{s}_n\widetilde{\alpha}$, $\dbar_n H=s_m \widetilde{\alpha}$; $d_{m+2}\dbar_{n+2} H=\omega$. Then $H$ is well-defined, for example
	\begin{align*}
		d_m(d_{m+2}\dbar_{n+2}H)&=d_m\omega=\overline{s}_n\rho_\alpha,\\
		d_{m+1}\dbar_{n+2} (d_m H) &=d_{m+1}\dbar_{n+2}\overline{s}_n \widetilde{\alpha}=\overline{s}_n(d_{m+1}\dbar_{n+1}\widetilde{\alpha})=\overline{s}_n\rho_\alpha.
	\end{align*}
	Hence (by \Cref{example-1}) $H$ extends to a map $H:\Delta^{m+2,n+2}\to PX$. Notice that the map $\sigma=\mathrm{ev}\circ H$ satisfies that $d_i \sigma =*$, $i<m$; $\dbar_j \sigma=*$, $j<n$; $d_m \sigma =\overline{s}_n\alpha$, $\dbar_n \sigma=s_m \alpha$; $d_{m+2}\dbar_{n+2} \sigma=\beta$. Therefore by definition we have
	\[
	[d_{m+1}\dbar_{n+1}\sigma]=[\alpha]\diamond[\beta].
	\]
	Moreover notice that $d_{m+1}\dbar_{n+1}H$ is a lift of $d_{m+1}\dbar_{n+1}\sigma$ and that $d_{m+1}\dbar_{n+1}H$ is constant on $\Lambda^{m+1}_{m+1}\boxtimes\Delta^{n+1}\cup \Delta^{m+1}\boxtimes\Lambda^{n+1}_{n+1}$, for example
	\[
	d_m(d_{m+1}\dbar_{n+1}H)=d_m\dbar_{n+1}(d_m H)=d_m\dbar_{n+1}\overline{s}_n \widetilde{\alpha}=d_m\widetilde{\alpha}=*.
	\]
	Hence by the definition of $\rho$ we have that
	\begin{align*}
		\rho([\alpha]\diamond[\beta])&=\rho[d_{m+1}\dbar_{n+1}\sigma]\\
		&=[d_{m+1}\dbar_{n+1}(d_{m+1}\dbar_{n+1}H)]\\
		&=[d_{m+1}\dbar_{n+1}(d_{m+2}\dbar_{n+2}H)]\\
		&=[d_{m+1}\dbar_{n+1}\omega]=[\hat{d}_1\hat{\dbar}_1\theta]\\
		&=[\rho_\alpha]\diamond [\rho_\beta].
	\end{align*}
	This completes the proof.
	\end{proof}

\subsection{Fibration long exact sequences}
\begin{definition}[fiber sequence]
	Let $p: (X,*)\to (Y,*)$ be a pointed fibration, that is $p$ is a fibration and $p$ preserves base point. Let $F$ be the fiber of $p$ at the base point defined by the pull-back diagram
	\[
	\begin{tikzcd}
		F\ar[r,"i"]\ar[d] & X\ar[d,"p"]\\
		*\ar[r] & Y
	\end{tikzcd}
	\]
	It follows that $F$ is fibrant, and that $F$ is canonically pointed since the base point of $X$ is contained in $F$.
	
	If $Y$ is fibrant, we say $F\xrightarrow{i} X\xrightarrow{p} Y$ is a (pointed) fiber sequence.
\end{definition}

The goal of this subsection is to establish a long exact sequence in homotopy for a fiber sequence. Let us fix a fiber sequence $F\xrightarrow{i} X\xrightarrow{p} Y$ throughout this subsection. Then we can define a connecting map $\del\delbar: \pi_{m+1,n+1}^\A(Y)\to \pi_{m,n}^\BC(F)$ in the same way as we defined the homomorphism $\rho$ in the previous subsection. Indeed, one replaces $\mathrm{ev}:PX\to X$ by the fibration $p: X\to Y$.

In fact the connecting map $\del\delbar$ factors as $\del\delbar=\tau\circ\rho$ where $\tau: \pi_{m,n}^\BC(\Omega Y)\to \pi_{m,n}^\BC(F)$ is the map defined as follows. Given $\alpha:\Delta^{m,n}\to \Omega Y$ rel $\delta\Delta^{m,n}$, we may view it as a map $\alpha: \Delta^{m,n}\times\Delta^{1,1}\to Y$ rel $\delta\Delta^{m,n}\times\Delta^{1,1}\cup \Delta^{m,n}\times\delta_\A\Delta^{1,1}$. Then consider the lifting problem
\[
\begin{tikzcd}
	\delta\Delta^{m,n}\times\Delta^{1,1}\cup \Delta^{m,n}\times\Lambda^{1,1}\ar[d]\ar[r,"*"] & X\ar[d,"p"]\\
	\Delta^{m,n}\times\Delta^{1,1}\ar[r,"\alpha"]\ar[ur,"\widetilde{\alpha}",dotted] & Y
\end{tikzcd}
\]
As usual the lifting $\widetilde{\alpha}$ exists and yields a well-defined map
\[
\tau: \pi_{m,n}^\BC(\Omega Y)\to \pi_{m,n}^\BC(F),\ [\alpha]\mapsto[\hat{d}_1\hat{\dbar}_1\widetilde{\alpha}]. 
\]

\begin{lemma}
	The map $\tau$ is a monoid homomorphism with respect to Bott-Chern products.
\end{lemma}
\begin{proof}
	We simply note that the construction of $\tau$ does not touch the Bott-Chern side.
\end{proof}

\begin{corollary} For $m\ge 1$, $n\ge 1$, the connecting map
	\[
	\del\delbar: \left(\pi_{m+1,n+1}^\A(Y),\diamond\right)\to \left(\pi_{m,n}^\BC(F),\star\right)
	\]
	is a homomorphism of monoids.
\end{corollary}
\begin{proof}
It suffices to show that $\del\delbar$ factors as $\pi_{m+1,n+1}^\A(Y)\xrightarrow{\rho}\pi_{m,n}^\BC(\Omega Y)\xrightarrow{\tau}\pi_{m,n}^\BC (F)$, we sketch the proof. Given $\alpha:\Delta^{m+1,n+1}\to Y$ rel $\delta_\A\Delta^{m+1,n+1}$, we can follow the definition of $\rho$ to find a map
$\widetilde{\alpha}: \Delta^{m+1,n+1}\to PY$ rel $\Lambda^{m+1,n+1}$ such that $\rho[\alpha]=[d_{m+1}\dbar_{n+1}\widetilde{\alpha}]$. Next, we may view $\widetilde{\alpha}$ as a map $\Delta^{m+1,n+1}\times\Delta^{1,1}\to Y$ rel $\Lambda^{m+1,n+1}\times\Delta^{1,1}\cup\Delta^{m+1,n+1}\times\delta_\A\Delta^{1,1}$ and then find a map $\theta:\Delta^{m+1,n+1}\times\Delta^{1,1}\to X$ rel $\Lambda^{m+1,n+1}\times\Delta^{1,1}\cup\Delta^{m+1,n+1}\times\Lambda^{1,1}$ lifting it along $p: X\to Y$. Finally one sees that $d_{m+1}\dbar_{n+1}\hat{d}_1\hat{\dbar}_1\theta$ represents both $\tau(\rho[\alpha])$ and $\del\delbar[\alpha]$.
\end{proof}

\begin{theorem}\label{homotopyLES}
	Let $F\xrightarrow{i} X\xrightarrow{p} Y$ be a fiber sequence. For $m\ge 0$, $n\ge 0$, there is an induced long exact sequence
	\[
	\cdots \to \pi_{m,n}^\BC(\Omega F)\xrightarrow{i_*} \pi_{m,n}^\BC(\Omega X) \xrightarrow{p_*} \pi_{m,n}^\BC(\Omega Y)\xrightarrow{\tau}\pi_{m,n}^\BC(F)\xrightarrow{i_*} \pi_{m,n}^\BC(X)\xrightarrow{p_*} \pi_{m,n}^\BC(Y).
	\]
\end{theorem}
\begin{remark}
	All the sets are pointed by the constant class $[*]$. By exactness, we mean the preimage of $[*]$ agrees with the image of the previous map. The maps are homomorphisms of monoids whenever the products are defined.
\end{remark}
\begin{proof}[Proof of \Cref{homotopyLES}]
	We only verify that $\ker \tau\subset\im p_*$. Given $\alpha:\Delta^{m,n}\to \Omega Y$ rel $\delta\Delta^{m,n}$, we follow the definition of $\tau$ to construct a map $\widetilde{\alpha}: \Delta^{m,n}\times\Delta^{1,1}\to X$ such that $\tau_\alpha=\hat{d}_1\hat{\dbar}_1\widetilde{\alpha}$ represents $\tau[\alpha]$.
	
	Assume now $[\tau_\alpha]=[*]\in \pi_{m,n}^\BC(F)$ i.e., there is a homotopy $h: \Delta^{m,n}\times\Delta^{1,1}\to F$ such that $\hat{d}_0 h=\hat{\overline{s}}_0\tau_\alpha$, $\hat{\dbar}_0 h=\hat{s}_0\tau_\alpha$ and $\hat{d}_1\hat{\dbar}_1 h=*$. Then we may combine $\widetilde{\alpha}$ and $ih$ to obtain a map $\theta:\Delta^{m,n}\times\Delta^{2,2}\to X$ rel $\delta\Delta^{m,n}\times\Delta^{2,2}$ such that $\hat{d}_0\theta=\hat{\overline{s}}_0\widetilde{\alpha}$, $\hat{\dbar}_0\theta=\hat{s}_0\widetilde{\alpha}$ and $\hat{d}_2\hat{\dbar}_2 \theta=i h$. Then $\sigma=\hat{d}_1\hat{\dbar}_1\theta$ is a loop in $X$. Moreover $p\theta$ implies that
	\[
	[p\sigma]=[p\widetilde{\alpha}]\diamond[pih]=[\alpha]\diamond[*]=[\alpha].
	\]
	This proves that $[\alpha]=p_*[\sigma]\in \im p_*$.
	\end{proof}
	
	\begin{corollary}
		Let $F\xrightarrow{i} X\xrightarrow{p} Y$ be a fiber sequence. For $m\ge 0$, $n\ge 0$, there is an induced exact sequence
	\[
	\pi_{m+1,n+1}^\A(F)\xrightarrow{i_*} \pi_{m+1,n+1}^\A(X) \xrightarrow{p_*} \pi_{m+1,n+1}^\A(Y)\xrightarrow{\del\delbar}\pi_{m,n}^\BC(F)\xrightarrow{i_*} \pi_{m,n}^\BC(X)\xrightarrow{p_*} \pi_{m,n}^\BC(Y).
	\]
	\end{corollary}
	\begin{proof}
		This follows from \Cref{homotopyLES} and \Cref{Aeppli-is-BC-of-loop}.
	\end{proof}

\section{Bisimplicial abelian groups}
\subsection{Dold-Kan correspondence}
We now treat bisimplicial abelian groups and compute their Bott-Chern and Aeppli homotopy groups. We show that they coincide with the Bott-Chern homology groups of the associated Moore bicomplexes which we now define.
\begin{definition}[Moore bicomplex]
	Let $A$ be a bisimplicial abelian group. The \textit{Moore bicomplex} of $A$ is the bicomplex $A_{\bullet,\bullet}=\bigoplus\limits_{m,n\ge 0} A_{m,n}$ with boundary operators
	\begin{align*}
		\del &=\sum_{i=0}^m (-1)^i d_i: A_{m,n}\to A_{m-1,n}\\
		\delbar& =\sum_{j=0}^n (-1)^j\dbar_j: A_{m,n}\to A_{m,n-1}
	\end{align*}
	Then from the simplicial relations $\del^2=0, \delbar^2=0$, and moreover $\del\delbar=\delbar\del$ since $\{d_i\}, \{\dbar_j\}$ commute\footnote{the operators $\del'=(-1)^m\del$ and $\delbar'=(-1)^n\delbar$ satisfy $(\del')^2=(\delbar')^2=\del'\delbar'+\delbar'\del'=0$.}. Abusing the notation, we will denote the Moore bicomplex of $A$ by $A$.
\end{definition}

\begin{definition}[normalized Moore bicomplex]
	Let $A$ be a bisimplicial abelian group. Let $N_\del A$ and $N_\delbar A$ be the sub-bicomplexes of $A$ given by
\begin{align*}
	(N_\del A)_{m,n} & =\cap_{i=0}^{m-1} \ker d_i\\
	(N_\delbar A)_{m,n} & =\cap_{j=0}^{n-1} \ker \dbar_j
\end{align*}
with boundary operators $((-1)^m d_m, \delbar)$ and $(\del, (-1)^n \dbar_n)$ respectively. Then we define the \textit{normalized Moore bicomplex} of $A$, denoted by $NA$, to be
	\[
	NA=N_\del A\cap N_\delbar A
	\]
	with restricted boundary operators $(-1)^m d_m$ and $(-1)^n \dbar_n$.
\end{definition}
\begin{lemma}\label{DK-homotopy}
	The inclusion $i: NA\to A$ is a chain homotopy equivalence of bicomplexes. Moreover, this equivalence is natural with respect to bisimplicial abelian groups $A$.
\end{lemma}
\begin{proof}
	Let $i_\del$ (resp. $i_\delbar$) be the inclusion of $N_\del A$ (resp. $N_\delbar A$) into $A$. Then by Dold-Kan, for each fixed $m,n$ there are chain maps
	\begin{align*}
		& f_\del: A_{\bullet, n}\to (N_\del A)_{\bullet,n}\\
		& f_\delbar: A_{m,\bullet}\to (N_\delbar A)_{m,\bullet}
	\end{align*}
	such that $f_\del\circ i_\del=1_{N_\del A}$, $ f_\delbar\circ i_\delbar=1_{N_\delbar A}$, and there are chain homotopies
	\begin{align*}
		& T_\del: A_{\bullet, n}\to A_{\bullet+1,n}\\
		& T_\delbar: A_{m,\bullet}\to A_{m,\bullet+1}
	\end{align*}
	such that $1_A-i_\del\circ f_\del=[\del, T_\del]$ and $1_A-i_\delbar\circ f_\delbar=[\delbar, T_\delbar]$. Here $[-,-]$ is the standard graded commutator.
	
	Since $f_\del,T_\del,f_\delbar,T_\delbar$ are natural in simplicial abelian groups, we have that $\{f_\del, T_\del\}$ commute with $\delbar$ and $\{f_\delbar, T_\delbar\}$ commute with $\del$. Therefore $f=f_\del\circ f_\delbar: A\to NA$ commutes with both $\del$ and $\delbar$, and $f\circ i=1_{NA}$. (Note that $NA=N_\del (N_\delbar A)=N_\delbar(N_\del A)$.) Meanwhile
	\[
	T=T_\del\circ T_\delbar: A_{m,n}\to A_{m+1,n+1}
	\]
defines a chain homotopy $i\circ f\simeq 1_A$, namely $1_A-i\circ f=[\del,[\delbar, T]]$.
\end{proof}

\begin{corollary}
	The inclusion $i: NA\to A$ induces isomorphisms on Bott-Chern, Aeppli and Dolbeault homologies (i.e. $\del$- and $\delbar$-homologies).
\end{corollary}
\begin{proof}
	Chain homotopic maps induce isomorphisms on these groups.
\end{proof}

\begin{definition}[degenerate sub-bicomplex]
	Let $A$ be a bisimplicial abelian group. Let $D_\del A_{m,n}$ and $D_\delbar A_{m,n}$ be the subgroups of $A_{m,n}$ given by
	\begin{align*}
		D_\del A_{m,n} & =\im(\oplus_{i=0}^{m-1} A_{m-1,n}\xrightarrow{\oplus s_i} A_{m,n})\\
		D_\delbar A_{m,n} &=\im(\oplus_{j=0}^{n-1} A_{m,n-1}\xrightarrow{\oplus \overline{s}_j} A_{m,n})
	\end{align*}
	Then $DA:=D_\del A+D_\delbar A$ forms a sub-bicomplex of the Moore bicomplex of $A$, called the \textit{degenerate sub-bicomplex}.
\end{definition}
\begin{lemma}\label{DK-iso}
	Let $A$ be a bisimplicial abelian group. Then the composite
	\[
		NA\xrightarrow{i} A\xrightarrow{p} A/DA
	\]
	is a natural isomorphism of bicomplexes.
\end{lemma}
\begin{proof}
By Dold-Kan, for all fixed $m,n$
\begin{align*}
	& (N_\del A)_{\bullet, n}\xrightarrow{i_\del} A_{\bullet, n}\xrightarrow{p_\del} (A/D_\del A)_{\bullet,n}\\
	& (N_\delbar A)_{m,\bullet}\xrightarrow{i_\delbar} A_{m,\bullet}\xrightarrow{p_\delbar} (A/D_\delbar A)_{m,\bullet}
\end{align*}
are natural isomorphisms of chain complexes. It follows that
\begin{align*}
	& N_\del A\xrightarrow{i_\del} A\xrightarrow{p_\del} A/D_\del A \\
	& N_\delbar A\xrightarrow{i_\delbar} A\xrightarrow{p_\delbar} A/D_\delbar A
\end{align*}
and consequently
\[
NA\xrightarrow{i} A\xrightarrow{p} A/DA
\]
are natural isomorphisms of bicomplexes.
\end{proof}

\begin{definition}[Eilenberg-MacLane space]
	Let $C_{\bullet,\bullet}$ be a $\NN\times\NN$-graded bicomplex. We define a bisimplicial set $K(C)$ by
	\[
	K(C)_{m,n}=\bigoplus\limits_{[m]\twoheadrightarrow [k]}\bigoplus\limits_{[n]\twoheadrightarrow [l]} C_{k,l}.
	\]
	which carries a natural bisimplicial structure induced by the face and degeneracy maps. We will call $K(C)$ the Eilenberg-MacLane space of $C$. 
\end{definition}

\begin{lemma}\label{DK-equivalence}
	The functors
	\[
	N: \mathbf{BsAb}\leftrightarrows \mathbf{BiCo_+} : K
	\]
	define an equivalence of categories.
\end{lemma}
\begin{proof}
	Apply the simplicial Dold-Kan correspondence in both horizontal and vertical directions.
\end{proof}

In summary, we have:
\begin{proposition}[Dold-Kan correspondence]\label{Dold-Kan} Let $A$ be a bisimplicial abelian group. Then we have the following:
\begin{enumerate}[(1)]
	\item The composition 
		\[
		NA\xrightarrow{i} A\xrightarrow{p} A/DA
		\]
		is an isomorphism of bicomplexes.
	\item The functors
	\[
	N: \mathbf{BsAb}\leftrightarrows \mathbf{BiCo_+} : K
	\]
	define an equivalence of categories.
	\item $i: NA\to A$ is a natural chain homotopy equivalence of bicomplexes.
\end{enumerate}
\end{proposition}
\begin{proof}
	These are \Cref{DK-iso}, \Cref{DK-equivalence} and \Cref{DK-homotopy}. 
\end{proof}

\subsection{Bott-Chern and Aeppli homotopy groups}
With Dold-Kan correspondence, we can compute Bott-Chern and Aeppli homotopy. Recall that the Bott-Chern and Aeppli homology of a bicomplex $C$ are
\[
H^\BC=\frac{\ker \del\cap\ker\delbar}{\im \del\delbar},\quad H^\A=\frac{\ker\del\delbar}{\im\del+\im\delbar}.
\]

\begin{proposition}\label{homotopy-vs-homology-BC} Let $A$ be a bisimplicial abelian group. Then there are natural group isomorphisms
\begin{align*}
	\pi_{m,n}^\BC(A,0) &\cong H_{m,n}^\BC(A)
\end{align*}
for $m\ge 0$, $n\ge 0$.
\end{proposition}
\begin{remark}
	The Bott-Chern homotopy set $\pi_{m,n}^\BC(A,0)$ inherits an abelian group structure from $A$ for all $m\ge 0$, $n\ge 0$. For $m\ge 1$ or $n\ge 1$, this group structure coincides with the one we previously studied by a standard Eckmann-Hilton argument. The isomorphism stated above is with respect to the group structures induced by $A$. The same applies to Aeppli homotopy.
\end{remark}

\begin{proof}[Proof of \Cref{homotopy-vs-homology-BC}]
	The set of maps $\alpha:\Delta^{m,n}\to A$ rel $\delta\Delta^{m,n}$ may be identified as the set
	\[
	\ker d_m\cap \ker \dbar_n\subset NA_{m,n}.
	\]
	The kernel of the projection
	\[
	\ker d_m\cap \ker \dbar_n\twoheadrightarrow\pi_{m,n}^\BC(A,0)
	\]
	may be identified with $\im (d_{m+1}\dbar_{n+1})$ by \Cref{character-BC}. This gives a natural isomorphism
	\[
	\pi_{m,n}^\BC(A,0)\cong H_{m,n}^{\BC}(NA).
	\]
	By Dold-Kan the latter is naturally isomorphic to $H_{m,n}^{\BC}(A)$.
\end{proof}
	
Next we compute Aeppli homotopy.
\begin{lemma}\label{BC-vanish-imply-A-vanish}
	Let $C$ be a bicomplex of abelian groups with $H_{m,n}^\BC(C)=0$ for all $m, n$. Then $H_{m,n}^\A(C)=0$ for all $m, n$.
\end{lemma}
\begin{proof}
	Let us consider 
	\[
	H_{m,n}^{\mathrm{dot}}:=\dfrac{\ker \del\cap \ker\delbar}{\im \del+\im\delbar}
	\] then there are natural exact sequences
	\begin{enumerate}[(1)]
		\item $0\to H_{m,n}^\mathrm{dot}(C)\to H_{m,n}^\A(C)\xrightarrow{(\del,\delbar)} H_{m-1,n}^\BC(C)\oplus H_{m,n-1}^\BC(C)$.
		\item $H_{m+1,n}^\A(C)\oplus H_{m,n+1}^\A(C)\xrightarrow{\del\oplus\delbar} H_{m,n}^\BC(C)\to H_{m,n}^\mathrm{dot}(C)\to 0$.
	\end{enumerate}
	Therefore $H_{m,n}^\BC(C)=0$ for all $m,n$ implies that $H_{m,n}^\mathrm{dot}(C)=0$ by (2) and then by (1) $H_{m,n}^\A(C)=0$ for all $m,n$.
\end{proof}

\begin{proposition}\label{homologyLES}
	Let $0\to A\to B\to C\to 0$ be a short exact sequence of bicomplexes. Then there is an induced exact sequence:
	\[
	H_{m+1,n+1}^\A(A)\to H_{m+1,n+1}^\A(B)\to H_{m+1,n+1}^\A(C)\xrightarrow{\del\delbar} H_{m,n}^\BC(A)\to H_{m,n}^\BC(B)\to H_{m,n}^\BC(C)
	\]
	for all $m,n$.
\end{proposition}
\begin{proof}
	The connecting homomorphism $\del\delbar$ is defined as follows. For $x\in C$ with $\del\delbar x=0$, we choose a lift $\widetilde{x}\in B$. Then $\del\delbar\widetilde{x}$ is zero when passed to $C$ and therefore arises from $A$. The connecting homomorphism is the map $[x]\mapsto [\del\delbar\widetilde{x}]$. The map is well-defined and the exactness of the sequence is an easy diagram chase.
\end{proof}

\begin{proposition}\label{homotopy-vs-homology-A} Let $A$ be a bisimplicial abelian group. Then there are natural group isomorphisms
\begin{align*}
	\pi_{m,n}^\A(A,0) &\cong H_{m,n}^\A(A)
\end{align*}
for $m\ge 1$, $n\ge 1$.
\end{proposition}
\begin{proof}
	Consider the path space fibration $\Omega A\to PA\to A$. Note that $PA$ and $\Omega A$ inherit bisimplicial abelian group structures from $A$. Moreover the fibration sequence induces a short exact sequence of Moore bicomplexes.
	
	Since $PA$ is contractible, for $m,n\ge 0$ we have $\pi_{m,n}^\BC(PA,0)=0$, hence $H_{m,n}^\BC(PA)=0$. By \Cref{BC-vanish-imply-A-vanish}, this forces $H_{m,n}^\A(PA)=0$ for all $m,n$. Therefore by \Cref{homologyLES}, \Cref{homotopy-vs-homology-BC} and \Cref{Aeppli-is-BC-of-loop}, we have
	\[
	H_{m+1,n+1}^\A(A)\cong H_{m,n}^\BC(\Omega A)\cong \pi_{m,n}^\BC(\Omega A,0)\cong \pi_{m+1,n+1}^\A(A,0)
	\]
	for $m,n\ge 0$. The result thus follows.
\end{proof}
\begin{remark}
	By \Cref{character-A} there is a natural epimorphism $H^\A_{m,n}(NA)\to \pi^\A_{m,n}(A,0)$, but the author currently does not have a direct proof of its injectivity.
\end{remark}

\begin{example}
	Let $C$ be a bicomplex concentrated in positive bidegrees. Then 
	\[
	\pi_{\bullet,\bullet}^\BC\left(K(C),0\right)=H_{\bullet,\bullet}^\BC(C), \ \pi_{\bullet,\bullet}^\A\left(K(C),0\right)=H_{\bullet,\bullet}^\A(C).
	\]
\end{example}

\section{Further discussions}
One can start transporting results in simplicial homotopy theory and pluripotential homotopy theory to the bisimplicial setting. Here we list some other problems that the author finds interesting.

The present work is inspired by recent developments at the interface of homotopy theory and complex geometry, particularly those emphasizing the roles of Bott-Chern and Aeppli cohomology and cohomotopy, including \cite{angella2013lemma, stelzig2025pluripotential, milivojevic2024bigraded, cirici2025structures}. Historically, Bott-Chern and Aeppli cohomology were originally introduced in the study of complex manifolds \cite{bott1965hermitian, aeppli1965cohomology}.
\begin{problem}[complex manifold]
	Is there a natural way to associate a bisimplicial set $\mathrm{Sing}^{hol}_{\bullet,\bullet}(X)$ to a complex manifold $X$ so that the Bott-Chern and Aeppli cohomology of $\mathrm{Sing}^{hol}_{\bullet,\bullet}(X)$, upon tensoring $\CC$, coincide with those of $X$?
\end{problem}
There are a priori obstructions because the Bott-Chern cohomology ring of $X$ may not be defined over $\ZZ$.

\begin{problem}[torsion]
	The rational part of the theory is expected to coincide with Stelzig's theory, which we know contains strictly finer information than the (simplicial-rational cdga) theory. Does the bisimplicial theory over $\ZZ$ contain more information in torsion than the simplicial theory?
\end{problem}

If the answer is positive, then can we use this extra torsion information to understand torsion in simplicial homotopy theory?

If the answer is negative, then it strongly suggests that every complex manifold should have a bisimplicial homotopy type modulo ambiguities from $\mathrm{Gal}(\RR/\QQ)$.

\begin{problem}[Serre duality]
	The bisimplicial theory appears to be the most natural place to discuss Serre duality over the integers. One can talk about Serre duality bisimplicial homotopy types, analogous to the discussion of Poincar\'e duality homotopy types. Then one can start building a surgery theory with Serre duality instead of Poincar\'e duality. Is this worth investigating?
\end{problem}
 A large part of classical surgery theory, at least up to topological or PL theory, is purely homotopy theoretical. The Fr\"olicher spectral sequence behaves similarly to a surgery procedure---duality is preserved along the way while some cohomology classes are annihilated.

\bibliographystyle{alpha}
\bibliography{ref}

\Addresses

\end{document}